\documentclass[12pt,leqno]{article}  
\usepackage{amssymb, amscd, amsmath, amsthm}

\def\beq{\begin{equation}}
\def\eeq{\end{equation}}

\theoremstyle{definition}
\newtheorem*{rema}{Remark}

\allowdisplaybreaks

\theoremstyle{theorem}

\newtheorem{theorem}{Theorem}
\newtheorem{lemma}{Lemma}

\newtheorem{corollary}{Corollary}

\theoremstyle{definition}
\theoremstyle{remark}
\newtheorem{remark}{Remark}

\def\ve{\varepsilon}

\title{A remark on density theorems for Riemann's zeta-function}

\author{by\\
J\'anos Pintz\thanks{Supported by the National Research Development and Innovation Office, NKFIH,  
KKP~133819.}}

\numberwithin{equation}{section}

\date{Dedicated to the memory of Jingrun Chen on the occasion of the 50\textsuperscript{th}
anniversary of the publication of his celebrated $(1+2)$ theorem on
Goldbach's conjecture}

\begin{document}

\maketitle

\section{}
\label{sec:1}



\footnotetext{Keywords and phrases: Riemann's zeta function, density hypothesis, density theorems.}

\footnotetext{2020 Mathematics Subject Classification: Primary 11M26, Secondary 11M06.}

The plausibility of the Riemann Hypothesis (RH) is supported (among others) by results of type
\beq
\label{eq:1.1}
N(\sigma, T) := \sum_{\substack{\zeta(\beta + i\gamma) = 0\\
\beta \geq \sigma, \ 0 \leq \gamma \leq T}} = o(N(T)) \ \text{ with }\ N(T) := \sum_{\substack{\zeta (\beta + i \gamma) = 0\\
0 \leq \gamma \leq T}} 1 \sim \frac{T}{2\pi} \log T
\eeq
valid for all $\sigma > 1/2$ as $T \to \infty$.
\eqref{eq:1.1} was first shown by Bohr and Landau in 1914 (\cite{BL1914}).
The first estimate of type
\beq
\label{eq:1.2}
N(\sigma, T) \ll_\varepsilon T^{A(\sigma)(1 - \sigma) + \ve} \ \text{ for any } \ \ve > 0 \ \text{ and } \ \sigma \geq 1/2
\eeq
was shown few years later by Carlson \cite{Car1920} with
\beq
\label{eq:1.3}
A(\sigma) \leq 4\sigma.
\eeq

It was Hoheisel \cite{Hoh1930} who first observed that such estimates lead to arithmetic consequences about
the difference of consecutive primes.
He proved the first approximation towards the famous conjecture that there exists always a prime between two consecutive squares.
This conjecture was characterised by Landau \cite{Lan1913} in his plenary talk at ICM1912 in Cambridge as one of the four main problems of the distribution of primes,
besides the Riemann Hypothesis.
Hoheisel \cite{Hoh1930} could show ($p_n$ denotes the $n$\textsuperscript{th} prime)
\beq
\label{eq:1.4}
p_{n + 1} - p_n \ll p_n^\theta \ \text{ with } \ \theta = 1 - \frac1{33\,000}.
\eeq
In the proof important role was played by Carlson's density theorem \eqref{eq:1.2}--\eqref{eq:1.3}.

Later it was realized that a uniform estimate of the form
\beq
\label{eq:1.5}
A(\sigma) \leq A \ \text{ for all } \ \sigma \geq \frac12
\eeq
yields (combined with a slightly better zero-free region than the classical one of de la Vall\'ee Poussin)
\beq
\label{eq:1.6}
p_{n + 1} - p_n \ll_\ve p_n^{1 - 1/A + \ve} \ \text{ for any } \ \ve > 0.
\eeq
In particular, the best possible estimate $A = 2$ would almost imply Landau's conjecture.
This is especially remarkable in light of the fact that even assuming the Riemann Hypothesis the best estimate we know is
\beq
\label{eq:1.7}
p_{n + 1} - p_n \ll p_n^{1/2} \log^C p_n
\eeq
(with $C = 2$ easily by the Riemann--Von Mangoldt explicit formula, with $C = 1$ by a deeper argument of Cram\'er \cite{Cra1921}).

This explains the significance of the Density Hypothesis (DH) which states
\beq
\label{eq:1.8}
N(\sigma, T) \ll T^{2(1 - \sigma)} \log^C T \ \text{ with some }\ C > 0 \ \text{ for all } \sigma \geq \frac12,
\eeq
or, in a slightly weaker form, using the notation \eqref{eq:1.2},
\beq
\label{eq:1.9}
A(\sigma) \leq 2 \ \text{ for all } \ \sigma \geq 1/2 \ \ \left(\Longleftrightarrow N(\sigma, T) \ll_\ve T^{2(1 - \sigma) + \ve}\right).
\eeq

The Riemann Hypothesis clearly implies the Density Hypothesis.
However, Ingham \cite{Ing1937} showed that also the Lindel\"of Hypothesis (LH)
\beq
\label{eq:1.10}
\mu\left(\frac12\right) = 0, \ \text{ where }\ \mu(\alpha) = \text{\rm inf} \left\{\mu; |\zeta(\sigma + it)| \leq T^\mu \text{ for } \sigma \geq \alpha, \, 1 < |t| \leq T\right\}
\eeq
implies DH.

In 1954 Tur\'an \cite{Tur1954} used his celebrated power-sum method \cite{Tur1953}, \cite{Tur1984} to give a different proof of Ingham's result that LH implies DH.
In the same work he came very close to breaking the DH in the vicinity of the boundary line $\sigma = 1$.
In the following we will use the notation:
\beq
\label{eq:1.11}
s = \sigma + it, \ \ \sigma = 1 - \eta, \ \ A(1 - \eta) = B(\eta).
\eeq
His result was with a small constant $c_1$ \cite{Tur1954}
\beq
\label{eq:1.12}
N(1 - \eta, T) \ll T^{2\eta + \eta^{1.14}} \log^6 T \ \text{ for } \ \eta < c_1,
\eeq
or with our new notation
\beq
\label{eq:1.13}
B(\eta) \leq 2 + \eta^{0.14} \ \text{ for } \eta < c_1.
\eeq

It was 16 years later when G. Hal\'asz and Tur\'an succeeded to break the DH \cite{HT1969}, that is, to prove (in the refined form appearing as Theorem 38.2 of Tur\'an's book \cite{Tur1984}) with the notation \eqref{eq:1.1}--\eqref{eq:1.2}, \eqref{eq:1.11}
\beq
\label{eq:1.14}
B(\eta) \leq 1.2 \cdot 10^5 \eta^{1/2}, \ \text{ more precisely } \
N(\sigma, T) < T^{1.2 \cdot 10^5 \eta^{3/2}} \log^C T.
\eeq

The proof was based on
\begin{itemize}
\item[(i)] Vinogradov's estimate $\mu(1 - \eta) \ll \eta^{3/2}$;
\item[(ii)] Tur\'an's power-sum method;
\item[(iii)] a simple but ingenious idea of Hal\'asz \cite{Hal1968}.
\end{itemize}

Soon after this Bombieri \cite{Bom1971} gave a different proof using ideas from the large sieve coupled with (i) and (iii).
Further, Montgomery has shown the DH for $\sigma \geq 9/10$, i.e. $\eta \leq 1/10$
(\cite{Mon1969}, \cite{Mon1971}).
Using another simple but ingenious idea Huxley \cite{Hux1972} (with a refinement of Montgomery) reached the DH for $\eta \leq 1/6$ and proved that
\beq
\label{eq:1.15}
A(\sigma) \leq 12/5\ \text{ for }\ \sigma > 1/2.
\eeq

The DH was shown for larger and larger ranges, until improving the result $\sigma > 11/14$ of Jutila \cite{Jut1977} Bourgain proved its validity \cite{Bou2000} for
\beq
\label{eq:1.16}
\sigma \geq 25/32 \Longleftrightarrow \eta \leq 7/32.
\eeq

Another direction of the research was to give strong density estimates for small values of $\eta$ (often especially for $\eta \to 0$).
In this direction the most important results were reached (in alphabetical order) by Bourgain, Ford, Heath-Brown, Huxley, Ivic, Jutila and Montgomery.

All these results use the large sieve and in some form the mentioned idea of Hal\'asz, which needs $\eta \leq 1/4$.
This is also the reason that they cannot give improvements of the classical zero-density theorem of Ingham
\cite{Ing1940}, $A(\sigma) \leq 3/(2 - \sigma)$, which is still the best today in the whole range
\beq
\label{eq:1.17}
\sigma \in \left[\frac12, \frac34\right] \Longleftrightarrow \eta \in \left[\frac14, \frac12\right].
\eeq

Recently I gave two alternative variants for proving the DH for small values of $\eta$.
In \cite{Pin2022} the goal was to reach a possibly simple proof while in \cite{Pin2023} to show the strongest possible result both in specific ranges of $\eta < 1/12$ and for $\eta \to 0$.
Both proofs were based on Vinogradov's method and Hal\'asz's idea.
However, in the second \cite{Pin2023} it was important to use recent deep results of Heath-Brown \cite{Hea2017}, further of Bourgain, Demeter and Guth \cite{BGD2016}.

\section{}
\label{sec:2}

The goal of the present work is
\begin{itemize}
\item[(i)] to give a possibly simple proof for all values of $\sigma > 3/4$, i.e. $\eta < 1/4$, which
\item[(ii)] breaks the DH for the non-negligible range $\sigma \geq 7/8$, i.e.\ $\eta \leq 1/8$, further gives
\item[(iii)] explicit density theorems for different ranges of all $\eta < 1/4$ with $B(\eta) \to 0$ as $\eta \to 0$ and
\item[(iv)] besides the (possibly indispensable) simple idea of Hal\'asz uses only classical knowledge of the theory of Riemann's zeta-function (classical in the sense that known since at least hundred years).
\end{itemize}

This means that we avoid the use of Tur\'an's method, the large sieve, Vinogradov's method, further mean and large value theorems.
However, to break the DH we need the estimate reached by the Weyl--Hardy--Littlewood method \cite{Lit1922} (see also (5.15.1) of the book of Titchmarsh \cite{Tit1951}, 2\textsuperscript{nd} edition, or the paper of Landau \cite{Lan1924}):
\beq
\label{eq:2.1}
\mu\left(1 - \frac1{2^\ell}\right) \leq \frac1{(\ell + 2)2^\ell}, \quad \ell \geq 0
\eeq
which is proved by Hardy and Littlewood in \cite{Lit1922} for $\ell \geq 1$, but it is also true for $\ell = 0$ $(\mu(0) = 1/2)$ by the functional equation and $\mu(1) = 0$.

Alternatively, the method of Van der Corput would yield results of similar type \cite{Cor1921}, \cite{Cor1922}.
However, the use of \eqref{eq:2.1} seems to be simpler, although Van der Corput's result
\beq
\label{eq:2.2}
\mu\left(1 - \frac{\ell}{2^\ell - 2}\right) \leq \frac1{2^\ell - 2} \qquad (\ell \geq 2)
\eeq
may yield even stronger estimates for $N(1 - \eta, T)$ in many ranges of $\eta$.

In order to formulate our results we need some more notation.
For $\ell \geq 1$ let
\begin{align}
\alpha_\ell &:= 1 - \frac1{2^\ell}, \qquad \sigma > \alpha_\ell \Longleftrightarrow \eta < \frac1{2^\ell},
\label{eq:2.3}\\
u_\ell &:= u_\ell(\eta) := \frac{\mu(\alpha_\ell)}{\sigma - \alpha_\ell} = \frac{\mu(\alpha_\ell)}{2^{-\ell} - \eta},
\label{eq:2.4}\\
v_\ell &:= v_\ell(\eta) := u_{\ell - 1} (2\eta)  = \frac{\mu(\alpha_{\ell - 1})}{2(2^{-\ell} - \eta)}.
\label{eq:2.5}
\end{align}
By the estimate \eqref{eq:2.1} of Hardy and Littlewood we have for $\eta < 1/2^\ell$
\beq
\label{eq:2.6}
u_\ell \leq \frac1{(\ell + 2)(1 - 2^\ell \eta)}, \qquad v_\ell \leq \frac1{(\ell + 1)(1 - 2^\ell \eta)}\ \text{ for }\ \ell \geq 1.
\eeq

\begin{theorem}
\label{th:1}
For $\ell \geq 1$, $\eta < \max(1/3, 2^{-\ell})$ we have
\beq
\label{eq:2.7}
B(\eta) \leq \max\bigl(4 u_\ell(\eta), 3v_\ell(\eta)\bigr).
\eeq
\end{theorem}

This implies a series of explicit density theorems for small values of $\eta$ and the conditional result of Hal\'asz and Tur\'an \cite{HT1969} that assuming (LH) we have $N(\sigma, T) \ll T^\ve$ for $\sigma > 3/4$, which is for $\sigma > 3/4$ a strong improvement of the theorem of Ingham \cite{Ing1937} stating $A(\sigma) = 2$ for $\sigma > 1/2$ on LH.

\begin{corollary}
\label{cor:1}
The Lindel\"of Hypothesis implies
\beq
\label{eq:2.8}
N(\sigma, T) \ll_\ve T^\ve \ \text{ for } \ \sigma > 3/4.
\eeq
\end{corollary}

\begin{proof}
Choosing $\ell = 2$, LH and the definition \eqref{eq:2.4}--\eqref{eq:2.5} yield $u_\ell(\eta) = v_\ell(\eta) = 0$ for $\eta < 1/4$.
\end{proof}

\begin{corollary}
\label{cor:2}
For $\ell \geq 2$, $\eta < \frac1{2^\ell}$ we have
\beq
\label{eq:2.9}
B(\eta) \leq \frac{4}{(\ell + 2)(1 - 2^\ell \eta)}.
\eeq
\end{corollary}

\begin{proof}
$4/(\ell + 2) \geq 3/(\ell + 1)$ for $\ell \geq 2$ with equality for $\ell = 2$.

Denoting by $\log_2 x$ the logarithm of base $2$ and choosing $C$ large ($C = C(\eta) \to \infty$ as $\eta \to 0$, but with $C(\eta) = o(\log 1/\eta)$), we obtain for $\eta \in \bigl[1/2^{\ell(1 + o(1))}, 1/2^{\ell + C}\bigr]$ from \eqref{eq:2.7}.
\end{proof}

\begin{corollary}
\label{cor:3}
$\displaystyle
B(\eta) \leq \frac{4}{(\ell + 2)(1 - 2^{-C})} \leq \frac{4(1 + o(1))}{\log_2(1/\eta)} =
\frac{4\log 2(1 + o(1))}{\log(1/\eta)}$.
\end{corollary}

This is naturally weaker than a result obtainable by Vinogradov's method (see \cite{Pin2023} and \cite{Hea2017}) but shows still $B(\eta) \to 0$ as $\eta \to 0$ beyond breaking DH in some interval for $\sigma = 1 - \eta$ close to the point $s = 1$.

However, Corollary \ref{cor:2} yields also interesting (although already known) consequences form small values of $\ell$ ($\ell = 2,3$, for example), i.e., for relatively large ranges of $\eta$.
On the other hand, they are often stronger than earlier results proved in a more complicated way (as 11.30 of Ivic \cite{Ivi1985}, for example).

\begin{corollary}
\label{cor:4}
$\displaystyle B(\eta) \leq \frac1{1 - 4\eta}$ for $\eta < 1/4$.
\end{corollary}

\begin{corollary}
\label{cor:5}
$\displaystyle B(\eta) \leq \frac4{5(1 - 8\eta)}$ for $\eta < 1/8$.
\end{corollary}

Corollary \ref{cor:4} breaks the DH for $\eta \leq 1/8$, i.e. $\sigma \geq 7/8$,
while Corollary \ref{cor:5} only for $\eta \leq 3/40$.
But Corollary \ref{cor:5} is sharper than Corollary \ref{cor:4} for $\eta < 1/24$.
Even the case $\ell = 1$ of Theorem \ref{th:1} yields

\begin{corollary}
\label{cor:6}
$\displaystyle B(\eta) \leq \frac{3}{2(1 - 2\eta)}$ \ for \ $\eta < 1/2$.
\end{corollary}

This breaks the DH also for $\eta \leq 1/8$, that is, for $\sigma \geq 7/8$, but it is weaker than Corollary \ref{cor:4} for $\eta < 1/8$ (sharper for $\eta > 1/8$).

\begin{remark}
\label{rem:1}
Corollary \ref{cor:6} is trivial for $\eta \geq 2/7$, i.e.\ $\sigma \leq 5/7$.
\end{remark}

\begin{remark}
\label{rem:2}
It is an interesting feature of Corollary \ref{cor:6} that it gives a non-trivial result even in the range
$\eta \in [1/4, 2/7]$ in contrary to most applications of Hal\'asz's method, although it is weaker there than the classical theorem of Carlson \cite{Car1920} or the one of Ingham \cite{Ing1940}.
\end{remark}

\begin{remark}
\label{rem:3}
It is also interesting to note that the case $\ell = 1$ would yield the inequality of Corollary \ref{cor:6} with any estimate $\mu(1/2) \leq 3/16$ and any estimate of type $\mu(1/2) \leq 1/4 - c_2$ $(c_2 > 0)$ would break the DH for $\eta < c_3(c_2)$ with $c_3(c_2) > 0$.
\end{remark}

Finally, similarly to Theorem \ref{th:1} we can give an alternative proof of the case (12.10) of Theorem 12.1 of Montgomery \cite{Mon1971} which states

\begin{theorem}
\label{th:2}
$\displaystyle{B(\eta) \leq \frac2{1 - \eta}}$.
\end{theorem}

\begin{remark}
\label{rem:4}
This represents a slight improvement of the Theorem of Tur\'an (cf.\ \eqref{eq:1.12}--\eqref{eq:1.13}) since
\beq
\label{eq:2.10}
\frac2{1 - \eta} = 2 + 2\eta + O(\eta^2) \ \text{ as } \ \eta \to 0.
\eeq
\end{remark}

\begin{remark}
\label{rem:5}
The proof of Theorem \ref{th:2} does not need for any $\ell$ the estimate \eqref{eq:2.1} of Hardy and Littlewood or the estimate of Van der Corput (apart from the ``trivial'' case $\mu(0) = 1/2$ following from the functional equation).
\end{remark}

\section{Proof of Theorem \ref{th:1}}
\label{sec:3}

We will consider a maximal number $K$ of zeros
$\varrho_j = \beta_j + i\gamma_j = 1 - \eta_j + i\gamma_j$ with
$\gamma_j \in [T/2, T]$, $|\gamma_\nu - \gamma_j| \geq 1$ for
$\nu \neq j$ $(j, \nu \in [1, K])$ and $\beta_j := 1 - \eta_j \geq \sigma := 1 - \eta$.
Let $\ve$ be sufficiently small positive, not necessarily the same at different occurrences,
$\ve < \ve_0(\eta, \ell, T)$. 
Analogously let $C$ be a constant, different at different occurrences with $C \leq C(\eta, \ell)$.
Further, let $\mu$ be the M\"obius function and with $u_\ell$ in \eqref{eq:2.4} let
\begin{gather}
\label{eq:3.1}
\eta < \min(1/3, 1/2^\ell),\ \ X = T^{\ve^2},\ \ Y = T^{u_\ell + \ve},\ \ Y_1 = e^{3}Y, \ \
\lambda = \log Y, \\
\mathcal L = \max(\lambda, \log T) \qquad
M_X(s) = \sum_{n \leq X} \mu(n) n^{-s}, \qquad a_n = \sum_{d\mid n, d \leq X} \mu(d).\nonumber
\end{gather}

In the proof we will use Perron's formula \cite{Per1908} in the following special form.

\begin{lemma}
\label{lem:1}
Let $s = \sigma_0 + t$, $1 \leq |t| \leq T$,  $\max(1/3, 1 - 2\eta) \leq \sigma_0 \leq 1$, $1\leq N \leq T$, $[N_1, N_2] = I(N) \subseteq [N, 2N]$
\beq
\label{eq:3.2}
S := \underset{n \in I(N)}{\sum}\rule{0pt}{12pt}\!\! ' n^{-s} = \frac1{2\pi i} \int\limits_{(1 - \sigma_0 + 1/\mathcal L)} \zeta(s + w) \frac{N_2^w - N_1^w}{w} dw
\eeq
where the dash means that the term $n^{-s}$ is counted by a factor $1/2$ if $n = N_1$ or $N_2$.
\end{lemma}

Deforming the way of integration along the horizontal lines $\text{\rm Im } w = \pm 2T$ from the segment of $\text{\rm Re }w = 1 - \sigma_0 + 1/\mathcal L$ to the vertical line segment $\text{\rm Re }w = \alpha_{\ell - 1} - \sigma_0$ $\text{\rm Im }w = [-2T, 2T]$ and counting the contribution of the pole at $w = 1 - s$, cutting the third term into two parts according to $\kappa \leq 0$, or $\kappa \geq 0$ we obtain
\begin{align}
\label{eq:3.3}
S &\ll \frac{N^{1 - \sigma_0}}{|t|} + \mathcal L N^{-(\sigma_0 - \alpha_{\ell - 1})} T^{\mu(\alpha_{\ell - 1}) + o(1)} + \max_{\alpha_{\ell - 1} - \sigma_0 \leq \kappa \leq 1 - \sigma_0 + 1/\mathcal L}
N^\kappa T^{\mu(\sigma_0 + \kappa) - 1 + o(1)} \\
&\ll \frac{N^{1 - \sigma_0}}{|t|} + T^{-\ve} + T^{\mu(\alpha_{\ell - 1}) - 1 + o(1)} +
T^{1 - \sigma_0 + \mu(\sigma_0) - 1 + o(1)} \ll \frac{N^{1 - \sigma_0}}{|t|} + T^{-\ve},\nonumber
\end{align}
by $\sigma_0 + \kappa \geq \alpha_{\ell - 1}$, $N\leq T$, $\mu(\sigma_0) \leq \sigma_0 - \ve$,
$\sigma_0 - \alpha_{\ell - 1} \geq \frac1{2^{\ell - 1}} - 2\eta = 2(2^{-\ell} - \eta)$ if
\beq
\label{eq:3.4}
N \gg T^{v_\ell + \ve} = T^{u_{\ell - 1}(2\eta) + \ve} \Longleftrightarrow N^{1/2^{\ell - 1}} - 2\eta \gg
T^{\mu(\alpha_{\ell - 1}) + \ve}.
\eeq

Finally we will use for any fixed $k$ the well known property of the generalized divisor function:
\beq
\label{eq:3.5}
\tau_k(n) = \sum_{n_1 n_2 \ldots n_k = n} 1 \ll n^{o(1)} \ \text{ for } \ n \to \infty.
\eeq

\begin{rema}
If $N \geq T$ then the simple Theorem 4.11 of \cite{Tit1951} (provable by simple partial summation) yields
\beq
\label{eq:3.6}
S \ll N^{1 - \sigma_0}\bigm/ |t| + O(N^{-\sigma_0}).
\eeq
\end{rema}

Similarly to (4.2) of \cite{Pin2023} we start with
\begin{align}
\label{eq:3.7}
I_j :&= \frac1{2\pi i} \int\limits_{(3)} \sum_{n = 1}^\infty \frac{a_n}{n^{s + \varrho_j}} \frac{e^{s^2/\mathcal L + \lambda s}}{s} ds = \frac1{2\pi i} \int\limits_{(3)} M_X (s + \varrho_j)\zeta (s + \varrho_j) \frac{e^{s^2/\mathcal L + \lambda s}}{s} ds\\
&= \frac1{2\pi i} \int\limits_{(\alpha_\ell - \beta_j)} M_X(s + \varrho_j) \frac{\zeta(s + \varrho_j)}{s} e^{s^2/\mathcal L + \lambda s} ds + O \left(\frac{\mathcal L Y^{\eta_j}}{|\gamma_j|} e^{-\gamma_j^2/\mathcal L}\right) \nonumber\\
&\ll X \int\limits_{-\infty}^\infty \frac{|\gamma_j + t|^{\mu(\alpha_\ell) + o(1)}}{1 + |t|} e^{-|\gamma_j + t|^2/\mathcal L} Y^{-(1/2^\ell - \eta)} dt + O \left(e^{-T^2/5\mathcal L}\right)\nonumber\\
&\ll \mathcal L T^{\mu(\alpha_\ell) + o(1)} Y^{-(1/2^\ell - \eta)} + O\left(e^{-T^2/5\mathcal L}\right) = o(1), \nonumber
\end{align}
where the error term represented the contribution of the pole of $\zeta$ at $s = 1 - \varrho_j = \eta_j - i\gamma_j$ and we used the trivial estimates $|M_X(s)| \leq X$, $M_X(1) \ll \mathcal L$.

On the other hand, we can evaluate the LHS of \eqref{eq:3.7} according to the value of $n$.
First we note that $a_n = 0$ for $1 < n \leq X$ and for $n = 1$ we can shift the line of integration to $\text{\rm Re }s = -4$.
The pole at $s = 0$ contributes $1$ and the integral is $O(\mathcal L Y^{-4}) = o(1)$.
Further, we can shift the line of integration to $\text{\rm Re }s = \mathcal L$ for $n > Y_1 = Ye^3$, obtaining by $|a_n| \leq \tau(n) \leq 2 \sqrt{n} \ll \left| n^{\varrho_j}\right|$,
$\sum\limits_{n > M} n^{-u} \ll M^{-(u - 1)}$ the relation
\beq
\label{eq:3.8}
\int\limits_{(\mathcal L)} \sum_{n \geq e^{\lambda + 3}} \frac{a_n}{n^{s + \varrho_j}} e^{s^2/\mathcal L + \lambda s} ds \ll e^{-(\lambda + 3)(\mathcal L - 1) + \lambda \mathcal L} \int\limits_{-\infty}^\infty
e^{(\mathcal L^2 - t^2)/\mathcal L} dt = o(1).
\eeq
Summarizing the above we get
\beq
\label{eq:3.9}
\sum_{x < n < Y_1} \frac{a_n}{n^{\varrho_j}} f(n) = 1 + o(1),
\eeq
where by $\lambda \ll \mathcal L$ we have for every $n \geq 1$
\begin{align}
\label{eq:3.10}
f(n) :&= \frac1{2\pi i} \int\limits_{(3)} \frac{e^{s^2/\mathcal L + (\lambda - \log n)s}}{s} ds = \frac1{2\pi i} \int\limits_{(1/\mathcal L)} \frac{e^{s^2/\mathcal L + (\lambda - \log n)s}}{s} ds\\
&\ll \int\limits_{-\infty}^\infty \frac{e^{-t^2/\mathcal L}}{|1/\mathcal L + it|} dt \ll \mathcal L \log \mathcal L.\nonumber
\end{align}

From this we obtain by a dyadic subdivision of $(X, Y_1)$ for some $U \in(X, Y_1)$ and $I(U) \subseteq [U, 2U]$
\beq
\label{eq:3.11}
\sum_{j = 1}^K \biggl|\sum_{n \in I(U)} a_n^* n^{-\varrho_j}\biggr| \gg \frac{K}{\mathcal L} \ \text{ with }\ a_n^* = a_n f(n).
\eeq

Our strategy is to raise the Dirichlet polynomial $\sum a_n^* n^{-s}$ with $n \in I(U)$ to a suitable integral power $h$ bounded by $\log Y_1/\log X \leq C\ve^{-2}$,
to reach a polynomial with $n \in \bigl[U^h, (2U)^h\bigr]$ where $h$, i.e., $U^h$ is minimal with the condition $U^h > T^{v_\ell + \varepsilon}$ in order to satisfy \eqref{eq:3.4} with $U = N$.
The resulting polynomial will have coefficients $b_h^* \ll \tau_h(n)(\log \mathcal L)^h \mathcal L^h$
by \eqref{eq:3.10} and $|a_n| \leq \tau(n)$.

Further, with a suitable value of $M \in [U^h, (2U)^h]$, \eqref{eq:3.11} can be substituted by
\beq
\label{eq:3.12}
\sum_{j = 1}^K \biggl|\sum_{n \in I(M)} b_n^* n^{-\varrho_j}\biggr| \gg \frac{K}{\mathcal L^h},
\eeq
using H\"older's inequality.
If $U \geq T^{v_\ell + \ve)/2} = T^{(u_{\ell - 1}(2\eta) + \ve)/2}$ we take $h = 2$, while for $X \leq U \leq T^{(v_\ell + \ve)/2}$
we can find an $h \in \mathbb Z^+$ with $U^h \in \Bigl(T^{v_\ell + \ve}, T^{3(v_\ell + \ve)/2}\Bigr)$.
Let us define now the numbers $\varphi_j$ with $|\varphi_j| = 1$ so that
\beq
\label{eq:3.13}
\biggl| \sum_{n \in I(M)} b_n^* n^{-\varrho_j} \biggr| = \varphi_j \sum_{n \in I(M)} b_n^* n^{-\varrho_j} \qquad (j = 1,2, \ldots, K)
\eeq
should hold.
Hal\'asz's idea is to square the LHS of \eqref{eq:3.13}, interchange the order of summation over $j$ and $n$ and use the Cauchy--Schwarz inequality for the sum when $n$ runs through elements of $I(M)$ with
\beq
\label{eq:3.14}
b_n^* n^{-\varrho_j} = b_n^* n^{-1/2} \cdot n^{-1/2 + \eta_j - i\gamma_j}\ \ \ (n \in I(M)).
\eeq

This gives from \eqref{eq:3.2}--\eqref{eq:3.5} and \eqref{eq:3.12}--\eqref{eq:3.13} separating in the second term the diagonal terms (those with $\nu = j$).
\begin{align}
\label{eq:3.15}
\frac{K^2}{\mathcal L^{2h}} &\ll \Biggl(\sum_{j = 1}^K \varphi_j \sum_{n \in I(M)} b_n^* n^{-\varrho_j}\biggr)^2 = \biggl(\sum_{n \in I(M)} b_n^* n^{-1/2} \sum_{j = 1}^K \varphi_j n^{-1/2 + \eta_j - i\gamma_j}\biggr)^2\\
&\ll \biggl(\sum_{n \in I(M)} \frac{|b_n^*|^2}{n}\biggr)
\biggl(\sum_{j = 1}^K \sum_{\nu = 1}^K \varphi_j \overline{\varphi}_\nu \sum_{n \in I(M)} \frac1{n^{1 - \eta_j - \eta_\nu + i(\gamma_j - \gamma_\nu)}}\biggr)\nonumber\\
&\ll T^{o(1)} \biggl(K(K - 1)T^{-\ve} + M^{2\eta} \sum_{j = 1}^K \sum_{\nu = 1}^K \frac1{|\gamma_j - \gamma_\nu|} + KM^{2\eta}\biggr) \nonumber\\
&\ll K^2 T^{-\ve/2} + KM^{2\eta} T^{o(1)}, \nonumber
\end{align}
leading by $M \asymp U_h^2 \leq \max\left(Y_1^2, T^{3v_\ell/2 + \ve}\right)$ to the final estimate
\beq
\label{eq:3.16}
\hspace*{20mm} K \ll M^{2\eta} T^{o(1)} \ll T^{\eta \max (4 u_\ell , 3 v_\ell) + \ve}.\hspace*{30mm}\hfill\square
\eeq

\medskip
The proof of Theorem \ref{th:2} is essentially the same as that of Theorem~\ref{th:1}.
The change is that in place of \eqref{eq:2.5} we choose now
\beq
\label{eq:3.17}
v_\ell = v_\ell(\eta) =: u_\ell(2\eta) \ \text{ and } \ \ell = 0.
\eeq
We remark that we can suppose $\eta < 1/3$ since otherwise $2/(1 - \eta) \geq 1/\eta$.
From \eqref{eq:3.17} we obtain similarly to \eqref{eq:3.16} by $\alpha_0 = 0$, $\mu(0) = 1/2$
\beq
\label{eq:3.18}
B(\eta) \leq \max(4 u_0(\eta), 3u_0(2\eta)) = \max \left(\frac2{1 - \eta}, \frac3{2(1 - 2\eta)}\right),
\eeq
i.e.,
\beq
\label{eq:3.19}
B(\eta) \leq \begin{cases}
\frac{2}{1 - \eta} &\text{ for } \eta \leq 1/5,\\
\frac{3}{2(1 - 2\eta)} &\text{ for } \eta \geq 1/5.
\end{cases}
\eeq
We note that together with Ingham's estimate (cf.\ (12.9) of Theorem 12.1 of \cite{Mon1971}) this yields
\beq
\label{eq:3.20}
A(\sigma) \leq 5/2,
\eeq
and so (with the special treatment of the immediate neighbourhood of $\sigma = 1$)
we obtain by \eqref{eq:1.5}--\eqref{eq:1.6} the result of Montgomery,
\beq
\label{eq:3.21}
p_{n + 1} - p_n \ll_\ve p_n^{3/5 + \ve}.
\eeq

{\small
\noindent
J\'anos Pintz\\
ELKH Alfr\'ed R\'enyi Mathematical Institute\\
H-1053 Budapest\\
Re\'altanoda u. 13--15.\\
Hungary\\
e-mail: pintz{@}renyi.hu}

\end{document}